\documentclass[a4paper,11pt,leqno]{article}

\usepackage[latin1]{inputenc}
\usepackage[T1]{fontenc} 
\usepackage[english]{babel}
\usepackage{verbatim}
\usepackage{calligra}
\usepackage{enumerate}
\usepackage{dsfont}
\usepackage{amsfonts}
\usepackage{amsmath}
\usepackage{amsthm}
\usepackage{amssymb}
\usepackage{mathtools}
\usepackage{mathrsfs}
\usepackage{color}
\usepackage{graphicx}
\usepackage{bm}
\usepackage{srcltx}

\usepackage{tikz}
\usepackage[retainorgcmds]{IEEEtrantools}

\usepackage{graphicx}		
\usepackage[hypcap=false]{caption}

\theoremstyle{definition}
\newtheorem{defi}{Definition}[section]

\newtheorem{rmk}[defi]{Remark}

\theoremstyle{plane}
\newtheorem{thm}[defi]{Theorem}

\newcommand{\tbf}{\textbf}
\newcommand{\tit}{\textit}

\newcommand{\tsl}{\textsl}

\newcommand{\veps}{\varepsilon}

\newcommand{\wtilde}{\widetilde}

\newcommand{\oline}{\overline}

\newcommand{\ra}{\rightarrow}

\newcommand{\g}{\gamma}
\renewcommand{\k}{\kappa}
\newcommand{\s}{\sigma}
\renewcommand{\t}{\tau}

\newcommand{\lam}{\lambda}
\newcommand{\de}{\delta}
\renewcommand{\o}{\omega}

\newcommand{\R}{\mathbb{R}}

\newcommand{\N}{\mathbb{N}}

\newcommand{\T}{\mathbb{T}}

\newcommand{\D}{\mathbb{D}}

\renewcommand{\div}{{\rm div}\,}

\newcommand{\al}{\alpha}
\newcommand{\bt}{\beta}

\allowdisplaybreaks

\def\d{\partial}
\def\div{{\rm div}\,}
\newcommand{\dd}{\d^2_{x}}

\begin{document}

\newcommand{\fra}[1]{\textcolor{blue}{[***FF: #1 ***]}}

\title{\textsc{\Large{\textbf{Finite time blow-up for some parabolic systems \\ arising in turbulence theory}}}}

\author{\normalsize\textsl{Francesco Fanelli}$\,^1\qquad$ and $\qquad$
\textsl{Rafael Granero-Belinch\'on}$\,^{2}$ \vspace{.5cm} \\
\footnotesize{$\,^{1}\;$ \textsc{Universit\'e de Lyon, Universit\'e Claude Bernard Lyon 1}}  \vspace{.1cm} \\
{\footnotesize \it Institut Camille Jordan -- UMR 5208}  \vspace{.1cm}\\
{\footnotesize 43 blvd. du 11 novembre 1918, F-69622 Villeurbanne cedex, FRANCE} \vspace{.3cm} \\
\footnotesize{$\,^{2}\;$ \textsc{Universidad de Cantabria}}  \vspace{.1cm} \\
{\footnotesize \it Departamento de Matem\'aticas, Estad\'istica y Computaci\'on}  \vspace{.1cm}\\
{\footnotesize Avda. Los Castros s/n, Santander, SPAIN} \vspace{.3cm} \\
\footnotesize{Email addresses: $\,^{1}\;$\ttfamily{fanelli@math.univ-lyon1.fr}}, $\;$
\footnotesize{$\,^{2}\;$\ttfamily{rafael.granero@unican.es}}
\vspace{.2cm}
}

\date\today

\maketitle

\subsubsection*{Abstract}
{\footnotesize 
We study a class of non-linear parabolic systems relevant in turbulence theory. Those systems can be viewed as simplified versions of the
Prandtl one-equation and Kolmogorov two-equations models of turbulence.

We restrict our attention to the case of one space dimension. We consider initial data for which the diffusion coefficients may vanish.
We prove that, under this condition, those systems are locally well-posed in the class of Sobolev spaces of high enough regularity,
but also that there exist smooth initial data for which the corresponding solutions blow up in finite time.

We are able to put in evidence two different types of blow-up mechanism.
In addition, the results are extended to the case of transport-diffusion systems, namely to the case when convection is taken into account.

}

\paragraph*{\small 2020 Mathematics Subject Classification:}{\footnotesize 35K65 
(primary);
35B44, 
35B65, 
76F60 
(secondary).}

\paragraph*{\small Keywords: }{\footnotesize non-linear parabolic systems; degeneracy; local well-posedness; finite
time blow-up; curvature; turbulence.}


\section{Introduction} \label{s:intro}
In \cite{Mie}, Mielke studied the following non-linear parabolic system of degenerate type:
\begin{equation} \label{eq:mielke}
\left\{\begin{array}{l}
       \d_tv\,-\,\div\left(\o^\alpha\,\nabla v\right)\,=\,0 \\[1ex]
       \d_t\o\,-\,\kappa_0\,\div\left(\o^\beta\,\nabla\o\right)\,=\,\o^\alpha\,\big|\nabla v\big|^2\,.
       \end{array}
\right.
\end{equation}
This system is set on $(t,x)\in\R_+\times\Omega$, where $\Omega$ is a smooth domain in $\R^d$, for $d\geq2$, and is supplemented with Neumann boundary conditions.
The two unknowns $v(t,x)$ and $\o(t,x)\geq0$ are scalar functions of time and space variables. The real numbers $\big(\alpha,\beta,\kappa_0\big)$
are positive parameters.

There are two main features which are apparent in system \eqref{eq:mielke}. The first one is the degeneracy of both parabolic equations
near the set $\big\{x\in\Omega\,\big|\;\o(t,x)=0\big\}$, which of course constitutes an obstacle in studying well-posedness theory and qualitative
properties of solutions. The second one is the mechanism for which the energy dissipated by the function $v$ \tsl{via}
the effects of viscosity feeds up the other unknown $\o$, through the term appearing on the right-hand side of its equation.

\medbreak
The main motivation for studying system \eqref{eq:mielke} comes from turbulence theory. In this context,
the function $v$ can be interpreted as a sort of shear velocity, whereas $\o\geq0$ represents the internal energy and $a(\omega)$ and $b(\omega)$ are eddie viscosities.
As a matter of fact, we recall that the standard approach to the (statistical) theory of turbulence starts by decomposing the velocity field $U$ of the fluid
as
$$
U\,=\, u\,+\,\wtilde{u}\,,
$$
where $ u$ is the mean part and $\wtilde u$ is the fluctuating part. 
This decomposition of the velocity field leads to the so-called \emph{Reynolds averaged Navier-Stokes equations}, or RANS: these are a cascade of partial
differential equations for $u$, $\wtilde u$ and their higher order correlations. However, the RANS are not closed and a
\emph{closure hypothesis} is needed.
Usually, this closure hypothesis takes the form of either one or two additional partial differential equations. One of the most famous examples of one-equation model of turbulence is the one derived by Prandtl \cite{Pra} (see also \cite{Bul_Lew-Mal}, \cite{CR-Lew} and \cite{lewandowski1997mathematical} for an extensive study):
\begin{equation} \label{eq:pra_d}
\left\{\begin{array}{l}
       \d_tu\,+\,(u\cdot\nabla) u\,+\,\nabla\pi\,-\,\nu\,\div\left(\sqrt{k}\,\D u\right)\,=\,f \\[1ex]
       \d_tk\,+\,u\cdot\nabla k\,-\,\bt_{1}\,\div\left(\sqrt{k}\,\nabla k\right)\,=\,\bt_{2}\,\sqrt{k}\,\big|\D u\big|^2\,-\,k\,\sqrt{k} \\[1ex]
       \div u\,=\,0\,.
       \end{array}
\right.
\end{equation}
Similarly, a famous example of a two-equations model of turbulence is the one due to Kolmogorov \cite{Kolm_42} (see also \cite{Spald} for an English translation of
that paper and further insights, see \cite{lewandowski1993existence} for related studies), namely
\begin{equation} \label{eq:kolm_d}
\left\{\begin{array}{l}
       \d_tu\,+\,(u\cdot\nabla) u\,+\,\nabla\pi\,-\,\nu\,\div\left(\dfrac{k}{\theta}\,\D u\right)\,=\,f \\[1ex]
       \d_t\theta\,+\,u\cdot\nabla\theta\,-\,\alpha_1\,\div\left(\dfrac{k}{\theta}\,\nabla\theta\right)\,=\,-\,\alpha_2\,\theta^2 \\[1ex]
       \d_tk\,+\,u\cdot\nabla k\,-\,\alpha_3\,\div\left(\dfrac{k}{\theta}\,\nabla k\right)\,=\,-\,k\,\theta\,+\,\alpha_4\,\dfrac{k}{\theta}\,\big|\D u\big|^2 \\[2ex]
       \div u\,=\,0\,.
       \end{array}
\right.
\end{equation}
In \eqref{eq:pra_d} and \eqref{eq:kolm_d}, the symbol $\D u\,=\,(Du+\nabla u)/2$ (where $Du$ is the Jacobian matrix of $u$ and $\nabla u\,=\,^t(Du)$
its transpose matrix) stands for the symmetric part of $Du$. The various parameters $\nu$, $\beta_{i}$ for $i=1,2$, and $\alpha_{j}$ for $j=1,2,3,4$, are positive
real numbers, whose precise value is empirical. We refer to \cite{Frisch}, \cite{Dav}, \cite{Les} and \cite{launder1972lectures} for more details on turbulence theory
and its mathematical models and to the introduction of \cite{F-GB} for more details and further references about system
\eqref{eq:kolm_d}\footnote{Notice that in \cite{F-GB}, as well as in all the classical references about turbulence, the unknown $\theta$ appearing
in system \eqref{eq:kolm_d} is called instead $\o$. Here we have decided to change the notation to avoid any confusion with the function $\o$ appearing
in \eqref{eq:mielke}, which should rather play the role of the quantity $k$ in \eqref{eq:pra_d} or of $k/\theta$ in \eqref{eq:kolm_d}.}.
We also point out that the Kolmogorov model \eqref{eq:kolm_d} can be seen as an ancestor of the more recent (but nowadays still classical)
$k$-$\varepsilon$ model of turbulence introduced by Launder and Spalding in \cite{launder1974}.

\medbreak
As explained in \cite{Mie} (see Section 8 therein), we easily see that system \eqref{eq:mielke} with $\alpha=\beta=1/2$ represents a reduced version of
the one-equation model of Prandtl, while system \eqref{eq:mielke} with $\alpha=\beta=1$ is a toy model of the two-equations model of Kolmogorov.
More in general, one could consider the family of equations
\begin{equation} \label{eq:gen}
\left\{\begin{array}{l}
       \d_tv\,-\,\div\big(a(\omega)\,\nabla v\big)\,=\,0 \\[1ex]
       \d_t\o\,-\,\kappa_0\,\div\big(b(\omega)\,\nabla\o\big)\,=\,a(\omega)\,\big|\nabla v\big|^2\,.
       \end{array}
\right.
\end{equation}
The general system \eqref{eq:gen} was broadly studied by Lewandowski \cite{lewandowski1997analyse}, see also
\cite{bernardi2009automatic}, \cite{lederer2007rans} and \cite{parietti1998quasi}.
It reduces to the model \eqref{eq:mielke} considered by Mielke when taking $a(\omega)=\omega^\alpha$ and $b(\omega)=\omega^\beta$, a choice which we will always
assume from now on.

In order to see the relation between turbulence theory and system \eqref{eq:mielke}, we observe that
the two systems \eqref{eq:pra_d} and \eqref{eq:kolm_d} present the same main features of Mielke's model,
namely degeneracy of the viscosity (dissipative) effect for vanishing $k$ and energy transfer mechanism from $u$ to $k$ \textsl{via} the
$\beta_{2}$ and $\alpha_4$ terms.
If the latter issue is a key property any model of turbulence should retain, the former represents instead a delicate point for the mathematical analysis.
As a matter of fact, if we restrict our attention to the Kolmogorov model \eqref{eq:kolm_d} for instance, previous results on well-posedness
in the framework of both weak and strong solutions have always either avoided to consider the possible vanishing of the function $k$ (see \cite{Mie-Nau},
\cite{Kos-Kub} and \cite{Kos-Kub_glob}, for instance), or imposed suitable conditions on the initial data to control the way
$k$ may get close to the $0$ value (see for instance \cite{Bul-Mal} in this direction).

Recently, in \cite{F-GB} we considered a one-dimensional reduction of system \eqref{eq:kolm_d}, where of course we suppressed the
divergence-free condition on the velocity field $u$ (which is a scalar field in $1$-D) and we erased the pressure term $\nabla\pi$ from the equations.
Our focus in \cite{F-GB} was to consider initial data for which $k_0$ may vanish in a generic way. We proved local well-posedness in $H^m$ spaces
for any $m\in\N$ with $m\geq 2$, and the existence of smooth initial data which give rise to solutions which blow up in finite time. 
Still in \cite{F-GB}, we introduced a toy-model of the considered $1$-D Kolmogorov two-equation model of turbulence, namely (with the notation
of the present paper) the system
\begin{equation} \label{eq:toy-kolm}
\left\{\begin{array}{l}
       \d_tv\,+\,v\,\d_xv-\,\d_x\left(\o\,\d_x v\right)\,=\,0 \\[1ex]
       \d_t\o\,+\,v\,\d_x\o\,-\,\kappa_0\,\d_x\left(\o\,\d_x\o\right)\,=\,\o\,\big|\d_xv\big|^2\,.
       \end{array}
\right.
\end{equation}
As a matter of fact, the main difficulties for proving local well-posedness and finite time singularity formation in the case of vanishing $k$
can be already seen in equations \eqref{eq:toy-kolm}, which are much simpler than the full model \eqref{eq:kolm_d}, even in the $1$-D reduction case.
Notice that system \eqref{eq:toy-kolm} is analogous to Mielke's system \eqref{eq:mielke} in $1$-D, when taking $\alpha=\bt=1$ and when keeping
track of the transport terms $v\d_xv$ and $v\d_x\o$ in the two equations.

\medbreak
The main concern of paper \cite{Mie} was to study several properties of system \eqref{eq:mielke} in the degenerate case, namely in the case where
$\o\geq0$ is allowed to vanish on a set of non-zero measure. In particular, in that paper the author studied the scaling laws related to system
\eqref{eq:mielke}, steady states and travelling fronts, existence of weak and very weak solutions for special ranges of the two parameters $\alpha$ and $\beta$.

In the present paper, we consider the one-dimensional reduction of system \eqref{eq:mielke}, which reads
\begin{equation} \label{eq:toy}
\left\{\begin{array}{l}
       \d_tv\,-\,\d_x\left(\o^\alpha\,\d_xv\right)\,=\,0 \\[1ex]
       \d_t\o\,-\,\kappa_0\d_x\left(\o^\beta\,\d_x\o\right)\,=\,\o^\alpha\,\big|\d_xv\big|^2\,,
       \end{array}
\right.
\end{equation}
together with its transport-diffusion counterpart (see system \eqref{eq:parab} in Section \ref{s:convective}).
We set system \eqref{eq:toy} on the one-dimensional torus
\[
\T\,:=\,[-\pi,\pi]/\sim\,,
\]
where the symbol $\sim$ denotes the equivalence relation which identifies the points $-\pi$ and $\pi$, and we supplement it
with initial conditions $\big(v,\o\big)_{|t=0}\,=\,\big(v_0,\o_0\big)$, where $\o_0\geq0$.

We aim at proving similar results to those shown in \cite{F-GB} for the $1$-D reduction of the Kolmogorov two-equation model of turbulence \eqref{eq:kolm_d}.
More precisely, we will prove: 
\begin{enumerate}[(i)]
 \item local in time well-posedness of system \eqref{eq:toy} in Sobolev spaces $H^m(\T)$, for $m\in\N$, $m\geq2$;
 \item existence of smooth data for which the corresponding solutions blow up in finite time.
\end{enumerate}
While the analysis for item (i) will be strongly inspired by the techniques used in \cite{F-GB}, the result about item (ii)
will be pretty different in spirit. Let us comment a bit more about this.

As just said, concerning local well-posedness, we will follow the approach of \cite{F-GB}. Namely, in order to control the vanishing
of the function $\omega_0$, we will impose
suitable regularity assumptions on $\sqrt{\o_0}$; correspondingly, in order to propagate such regularity, we will work with the good
unknowns $\big(v,\sqrt\o\big)$ of system \eqref{eq:toy}, instead of the original unknowns $\big(v,\o\big)$.
Notice however that, although in \cite{Mie} the author is able to consider small values of $\alpha$, namely $0<\alpha<1$, for proving existence of weak solutions,
here we will need to impose rather the conditions $\alpha\geq1$ and $\bt\geq1$: these constraints arise naturally when performing
higher order energy estimates for $v$ and $\o$, in order to absorb some terms by using the degenerate viscosity (diffusion) effect. Let us mention that it is hard to say whether there is a classical solution for parameters outside of our results. It seems clear to us that a more detailed analysis or even new ideas are required to expand the region of parameters where a well-posedness result is available. With this in mind, it is possible to see that the change of unknown $\eta=\sqrt{\o}$ is the best possible choice: any other change of variable of the form $\s_n=\o^{1/n}$,
for some $n\in\N\setminus\{0\}$, would lead to more stringent contraints on the values of the parameters $\alpha$ and $\beta$.

Concerning finite time singularity formation, instead, our result will be quite different from that in \cite{F-GB}. In fact, loosely speaking,
in that paper the blow-up mechanism could be defined ``of Burgers type''. More precisely, the initial velocity field was assumed to be odd with respect to the origin, with
negative slope at $x=0$; this fact, together with the vanishing of $\o_0$ at the same point $x=0$ (a property which is preserved by the flow) was at the basis
of the singularity formation in finite time. On the contrary, in the present paper we will highlight a different blow-up mechanism, based on the
growth of the curvature $\d_x^2\o$ of the function $\o$. In the case $\bt>1$ we will see that $\dd\o$ blows up if and only if the slope
of the velocity fields becomes unbounded; whenever $\bt=1$, instead, one may have that $\d_xv(t,0)$ remains bounded, while $\dd\o(t,0)$ blows up
in finite time. A similar argument applies also to the case of the convection-diffusion system (see Section \ref{s:convective}).

Finally, let us also emphasize that the singularity formation appears related to a point where the variable $\omega=0$. In this regards, it seems to be related with regions of vacuum or cavitation.

\medbreak
We conclude this introduction with an overview of the contents of the paper.

In the next section, we deal with the problem of the local in time well-posedness of system \eqref{eq:toy}. The main result in that direction is Theorem \ref{t:wp}.
In Section \ref{s:blow-up}, instead, we state and prove our main result concerning singularity formation, namely
Theorem \ref{t:symmetry}. Finally, in Section \ref{s:convective} we adapt our discussion about well-posedness and finite time blow-up
to cover the case when a convection term is added to both equations in \eqref{eq:toy}, see system \eqref{eq:parab}. The main results of that section
are contained in Theorems \ref{th:sing_convective} and \ref{th:sing_conv-2}.

\section{Local well-posedness} \label{s:wp}

This section is devoted to stating and proving local well-posedness for system \eqref{eq:toy}. To begin with, we decide to follow the approach
of \cite{F-GB} and work with the set of good unknowns $\big(v,\sqrt{\o}\big)$. We will generalise this approach in Subsection \ref{ss:general} below.

Our first main result of this section can be formulated in the following way.

\begin{thm} \label{t:wp}
Fix $m\in\N$ such that $m\geq 2$. Take an initial datum $\big(v_0,\o_0\big)$ belonging to $H^m(\T)$, such that $\o_0\geq0$
and $\sqrt{\o_0}\in H^m(\T)$ as well. Assume in addition that
\begin{align*}
&\al\,\in\,\left\{1\,,\ \frac{3}{2}\,, \ 2\,, \ \ldots \frac{m}{2}\right\}\qquad\qquad \mbox{ or }\quad \al\,\geq\,\frac{m+1}{2}\,, \\
&\beta\,\in\,\left\{1\,,\ \frac{3}{2}\,, \ 2\,, \ \ldots \frac{m}{2}\right\}\qquad\qquad \mbox{ or }\quad \beta\,\geq\,\frac{m+1}{2}\,,
\end{align*}

Then, there exists a time $T>0$ such that system \eqref{eq:toy} admits a unique solution $\big(v,\o\big)$ on $[0,T]\times\T$
with the following properties:
\begin{itemize}
\item for any time $t\in[0,T]$, one has $\o(t)\geq0$;
\item $v$ and $\sqrt\o$ belong to $L^\infty\big([0,T];H^m(\T)\big)\,\cap\, \bigcap_{s<m}C\big([0,T];H^s(\T)\big)$;
\item the functions $\o^{\alpha}\,\d_x^{m+1}v$ and $\o^{\beta}\,\d_x^{m+1}\sqrt{\o}$ both belong to $L^2\big([0,T];L^2(\T)\big)$.
\end{itemize}
\end{thm}

The proof of the previous theorem follows the main lines of Theorem 1.1 in \cite{F-GB}. Therefore, here we limit ourselves to state
\tsl{a priori} estimates for smooth solutions, leaving as an exercise for the reader
the precise construction of smooth approximate solutions satisfying the uniform bounds and the proof of the convergence of those solutions to a true solution
of system \eqref{eq:toy}.

Before moving on, once for all in this section we fix the value
\[
\kappa_0\,=\,1\,.
\]
However, we will keep track of the precise value of this parameter in the blow-up result, see Sections \ref{s:blow-up} and \ref{s:convective},
devoted respectively to the purely parabolic case, \tsl{i.e.} system \eqref{eq:toy}, and to the system with convective term.

\medbreak
The proof of \tsl{a priori} estimates is divided into two steps. First of all, we present estimates for the lower order norms, namely $L^p$-type norms
of $v$ and $\o$. In the second step (performed in Subsection \ref{ss:high}), instead, we exhibit bounds for the higher order norms, namely $L^2$ norms of
the derivatives $\d_x^mv$ and $\d_x^m\o$.

\subsection{Estimates for the $L^p$ norms of the solution} \label{ss:L^p}

We present here $L^p$-type \tsl{a priori} bounds for solutions to system \eqref{eq:toy}. Recall that, throughout this section, we assume to
have smooth solutions on $(t,x)\in\R_+\times\T$, so all the computations below will be fully justified.

First of all, from the equation for $\o$ we see that
\[
 \d_t\o\,-\,\d_x\big(\o^\bt\,\d_x\o\big)\,\geq\,0\,.
\]
Then, by \textsl{e.g.} comparison with the smooth solution $w$ to the porous medium equation $\d_tw\,-\,\d_x\big(w^\bt\,\d_xw\big)\,=\,0$ (see Chapter 3 of \cite{Vaz}),
$\o$ remains positive for all times:
\begin{equation} \label{est:o_pos}
 \forall\,(t,x)\in\R_+\times\T\,,\qquad\qquad \omega(t,x)\,\geq\,0\,.
\end{equation}

Next, we perform an energy estimate for $v$, obtaining that
\begin{equation} \label{est:v-L^2}
\forall\,t\geq0\,,\qquad\qquad \|v(t)\|_{L^2}^2\,+\,2\int^t_0\int_\T\o^\alpha\,\big|\d_xv\big|^2\,dx\,d\t\,\leq\,\left\|v_0\right\|_{L^2}^2\,.
\end{equation}
On the other hand, integrating the equation for $\o$ on the torus, we get
\[
\frac{1}{2}\,\frac{d}{dt}\int_\T\o\,dx\,+\,\int_\T\o^\bt\,\big|\d_x\o\big|^2\,dx\,=\,\int_\T\o^\alpha\,\big|\d_xv\big|^2\,dx\,.
\]
Thus, in view of \eqref{est:o_pos} and \eqref{est:v-L^2}, an integration in time yields
\begin{equation} \label{est:o_L^1}
\forall\,t\geq0\,,\qquad\qquad
\|\o(t)\|_{L^1}\,+\,\int^t_0\int_\T\o^\bt\,\big|\d_x\o\big|^2\,dx\,d\t\,\leq\,\left\|\o_0\right\|_{L^1}\,+\,\left\|v_0\right\|_{L^2}^2\,.
\end{equation}
Observe that, after defining
\[
\eta_0\,:=\,\sqrt{\o_0}\qquad\qquad \mbox{ and }\qquad\qquad \eta\,:=\,\sqrt{\o}\,,
\]
inequality \eqref{est:o_L^1} yields
\begin{equation} \label{est:eta_L^2}
\forall\,t\geq0\,,\qquad\qquad
\|\eta(t)\|_{L^2}^2\,\leq\,\left\|\eta_0\right\|_{L^2}^2\,+\,\left\|v_0\right\|_{L^2}^2\,.
\end{equation}

\subsection{Propagation of high regularity norms} \label{ss:high}

We now turn our attention to the estimate for the high regularity norm. For the sake of conciseness, we restrict our attention to $H^2$ estimates; propagation
of higher order regularities easily follow by analogous arguments.

In addition, we observe that, for any smooth function $u\in H^2(\T)$, for any $p\in[1,+\infty]$ one has
\[
\left\|u\right\|^2_{H^2}\,\lesssim\,\|u\|_{L^p}^2\,+\,\left\|\dd u\right\|_{L^2}^2\,,
\]
see the details \tsl{e.g.} in \cite{F-GB}. Therefore, we can focus only on the propagation of the $L^2$ bounds for $\dd v$ and $\dd\eta$.

To begin with, we formulate system \eqref{eq:toy} in terms of the new set of unknowns $\big(v,\eta\big)$: we have
\begin{equation} \label{eq:toy_good}
\left\{\begin{array}{l}
       \d_tv\,-\,\d_x\left(\eta^{2\alpha}\,\d_xv\right)\,=\,0 \\[1ex]
       \d_t\eta\,-\,\d_x\left(\eta^{2\beta}\,\d_x\eta\right)\,=\,\dfrac{1}{2}\,\eta^{2\alpha-1}\,\big|\d_xv\big|^2\,+\,\eta^{2\bt-1}\,\big|\d_x\eta\big|^2\,.
       \end{array}
\right.
\end{equation}

Then, we derive equations for $\dd v$ and $\dd\eta$.  Let us start with $\dd v$: differentiating twice the first equation in \eqref{eq:toy_good}, we obtain
\begin{equation} \label{eq:dd_v}
\d_t\dd v\,-\,\d_x\big(\eta^{2\al}\,\d_x\dd v\big)\,=\,f\,,
\end{equation}
where we have defined
\begin{align*}
f\,&:=\,\d_x\Big(2\d_x(\eta^{2\al})\,\dd v+\d_x^2(\eta^{2\al})\,\d_x v\Big) \\
&=\,2\al\,(2\al-1)\,(2\al-2)\,\eta^{2\al-3}\,(\d_x\eta)^3\,\d_xv\,+\,6\al\,(2\al-1)\,\eta^{2\al-2}\,\d_x\eta\,\dd\eta\,\d_xv \\
&\; +\,6\al\,(2\al-1)\,\eta^{2\al-2}\,(\d_x\eta)^2\,\dd v\,
+\,2\al\,\eta^{2\al-1}\,\d_x^3\eta\,\d_xv\,+\,6\al\,\eta^{2\al-1}\,\dd\eta\,\dd v\,+\,4\al\,\eta^{2\al-1}\,\d_x\eta\,\d_x^3v \\
&=\,\sum_{j=1}^6f_j\,.
\end{align*}
Likewise, we find the following equation for $\dd\eta$:
\begin{equation} \label{eq:dd_eta}
\d_t\dd \eta\,-\,\d_x\big(\eta^{2\bt}\,\d_x\dd \eta\big)\,=\,g\,,
\end{equation}
where, this time, we have set
\begin{align*}
g\,&:=\,\d_x\Big(2\d_x(\eta^{2\bt})\,\dd \eta+\d_x^2(\eta^{2\bt})\,\d_x\eta\Big)\,+\,
\dd\left(\dfrac{1}{2}\,\eta^{2\alpha-1}\,\big|\d_xv\big|^2\,+\,\eta^{2\bt-1}\,\big|\d_x\eta\big|^2\right) \\
&=\,(2\bt+1)\,(2\bt-1)\,(2\bt-2)\,\eta^{2\bt-3}\,(\d_x\eta)^4\,+\,(12\bt+5)\,(2\bt-1)\,\eta^{2\bt-2}\,(\d_x\eta)^2\,\dd\eta \\
&\qquad\; +\,(6\bt+2)\,\eta^{2\bt-1}\,\big(\dd\eta\big)^2\,+\,(6\bt+2)\,\eta^{2\bt-1}\,\d_x\eta\,\d_x^3\eta \\
&\qquad\; +\,(2\al-1)\,(\al-1)\,\eta^{2\al-3}\,\big(\d_x\eta\big)^2\,\big(\d_xv\big)^2\,
+\,\dfrac{(2\al-1)}{2}\,\eta^{2\al-2}\,\dd\eta\,\big(\d_xv\big)^2 \\
&\qquad\; +\,2\,(2\al-1)\,\eta^{2\al-2}\,\d_x\eta\,\d_xv\,\dd v\,+\,\eta^{2\al-1}\,\big(\dd v\big)^2\,+\,\eta^{2\al-1}\,\d_xv\,\d_x^3v\\
&=\,\sum_{j=1}^9g_j\,.
\end{align*}

We now perform energy estimates on the previous equations \eqref{eq:dd_v} and \eqref{eq:dd_eta}. 
In the computations, we will need the following interpolation inequalities (see \tsl{e.g.} Section 2 of \cite{F-GB} for details).
First of all,
\begin{equation} \label{est:H2-interp}
\forall\,u\in H^2(\T)\,, 
\qquad\qquad
\left\|\d_xu\right\|_{L^2}\,\lesssim\,\left\|u\right\|_{L^2}^{1/2}\;\left\|\dd u\right\|_{L^2}^{1/2}\,.
\end{equation}

Moreover, we have
\begin{equation} \label{est:interpolation}
\forall\,u\in H^2(\T)\,, 
\qquad\qquad
\left\|\d_x u\right\|_{L^4}\,\lesssim\,\left\|u\right\|_{L^\infty}^{1/2}\;\left\|\dd u\right\|_{L^2}^{1/2}\,.
\end{equation}
In fact, notice that we could replace $u$ by $u-\oline u$ in \eqref{est:H2-interp} and \eqref{est:interpolation}, where $\oline u$ denotes
the mean value of $u$ on the torus $\T$: after setting $|\T|$ to be the Lebesgue measure of $\T$, we have
\[
\oline u\,:=\,\frac{1}{|\T|}\int_\T u(x)\,dx\,.
\]

Finally, we have
\begin{equation} \label{est:interp_2}
 \forall\,u\in H^1(\T)\quad \mbox{ such that }\;\oline u\,=\,0\,,\qquad\qquad
\left\|u\right\|_{L^\infty}\,\lesssim\,\left\|u\right\|_{L^p}^{p/(p+2)}\;\left\|\d_xu\right\|_{L^2}^{2/(p+2)}\,,
\end{equation}
for any $p\in[1,+\infty]$.

\subsubsection{Estimates for $\dd v$} \label{sss:v-H^2}

We start by considering the function $v$. Multiplying equation \eqref{eq:dd_v} by $\dd v$ and integrating by parts, standard computations
yield
\begin{equation} \label{est:en-dd_v}
\frac{1}{2}\,\frac{d}{dt}\left\|\dd v\right\|_{L^2}^2\,+\,\int_\T\eta^{2\al}\,\left|\d_x^3v\right|\,dx\,=\,\int_\T f\,\dd v\,dx\,.
\end{equation}
Thus, we have to bound the $L^2$ scalar product of each function $f_j$ with $\dd v$, for any $j\in\{1\ldots 6\}$: this is our next goal.

We start by considering the term $f_1$. Recall that, by assumption, either $\alpha=1$, in which case $f_1=0$, or $2\alpha-3\geq0$. If this latter
condition holds, we can bound
\begin{align*}
\left|\int_\T f_1\,\dd v\,dx\right|\,&\lesssim\,\left\|\eta\right\|_{L^\infty}^{2\al-3}\,\left\|\big(\d_xv,\d_x\eta\big)\right\|_{L^\infty}^2\,
\left\|\d_x\eta\right\|_{L^4}^2\,\left\|\dd v\right\|_{L^2} \\
&\lesssim\,\left\|\eta\right\|_{L^\infty}^{2\al-2}\,\left\|\big(\d_xv,\d_x\eta\big)\right\|_{L^\infty}^2\,
\left\|\big(\dd v,\dd\eta\big)\right\|_{L^2}^2\,,
\end{align*}
where, in passing from the first line to the second one, we have used inequality \eqref{est:interpolation}.

The terms $f_2$ and $f_3$ are bounded in a pretty analogous way:
\begin{align*}
\left|\int_\T \big(f_2+f_3\big)\,\dd v\,dx\right|\,&\lesssim\,
\left\|\eta\right\|_{L^\infty}^{2\al-2}\,\left\|\big(\d_xv,\d_x\eta\big)\right\|_{L^\infty}^2\,\left\|\big(\dd v,\dd\eta\big)\right\|_{L^2}^2\,.
\end{align*}

Now, let us consider the term involving $f_6$. We can estimate it by using Cauchy-Schwarz and Young inequalities: we have
\begin{align*}
\left|\int_\T f_6\,\dd v\,dx\right|\,&\lesssim\,\left\|\eta\right\|_{L^\infty}^{\al-1}\,\left\|\d_x\eta\right\|_{L^\infty}\,\left\|\dd v\right\|_{L^2}\,
\left\|\eta^\alpha\,\d_x^3v\right\|_{L^2} \\
&\leq\,C_\de\,\left\|\eta\right\|_{L^\infty}^{2\al-2}\,\left\|\d_x\eta\right\|^2_{L^\infty}\,\left\|\dd v\right\|_{L^2}^2\,+\,\de\,
\left\|\eta^\alpha\,\d_x^3v\right\|^2_{L^2}\,,
\end{align*}
where the parameter $\de>0$ is, for the time being, arbitrary and will be fixed small enough later on, and the multiplicative constant $C_\de>0$
depends also on such $\de$.

Next, we control the term involving $f_5$. By integrating by parts once, we reconduct ourselves to bounding terms of the same type as the previous ones:
\[
\int_\T f_5\,\dd v\,dx\,=\,-6\al\,(2\al-1)\int_\T\eta^{2\al-2}\,\big(\d_x\eta\big)^2\,\big(\dd v\big)^2\,dx\,-\,
12\al\int_\T\eta^{2\al-1}\,\d_x\eta\,\dd v\,\d_x^3v\,dx\,.
\]
Therefore, we get
\begin{align*}
\left|\int_\T f_5\,\dd v\,dx\right|\,&\leq\,C_\de\,\left\|\eta\right\|_{L^\infty}^{2\al-2}\,\left\|\d_x\eta\right\|_{L^\infty}^2\,
\left\|\dd v\right\|_{L^2}^2\,+\,\de\,\left\|\eta^\alpha\,\d_x^3v\right\|^2_{L^2}\,,
\end{align*}
where $\de>0$ is the same as above and the (new) constant $C_\de>0$ depends on it.

Finally, as for $f_4$, it is better to avoid the use of the $L^2$ norm of $\eta^\beta\,\d_x^3\eta$ (which would cause a more restrictive condition
on the parameters); we rather integrate by parts, to obtain
\begin{align*}
\int_\T f_4\,\dd v\,dx\,&=\,-\,2\al\,(2\al-1)\int_\T\eta^{2\al-2}\,\d_x\eta\,\dd\eta\,\d_xv\,\dd v\,dx \\
&\qquad\qquad -\,2\al\int_\T\eta^{2\al-1}\,\dd\eta\,\dd v\,dx\,-\,2\al\int_\T\eta^{2\al-1}\,\dd\eta\,\d_xv\,\d_x^3 v\,dx\,.
\end{align*}
Thus, it is easy to get, for any $\de>0$ as above and a corresponding new constant $C_\de>0$, depending only on $\de$ and $\alpha$, the
following estimate:
\begin{align*}
\left|\int_\T f_4\,\dd v\,dx\right|\,&\leq\,C_\de\,\left\|\eta\right\|_{L^\infty}^{2\al-2}\,\left\|\big(\d_xv,\d_x\eta\big)\right\|_{L^\infty}^2\,
\left\|\big(\dd v,\dd\eta\big)\right\|_{L^2}^2\,+\,2\de\,\left\|\eta^\alpha\,\d_x^3v\right\|^2_{L^2}\,.
\end{align*}

We can now insert all those bounds into \eqref{est:en-dd_v}: taking $\de>0$ small enough, in order to absorbe the corresponding terms on the left-hand side
of that expression, we finally find
\begin{equation} \label{est:dd_v_final}
 \frac{d}{dt}\left\|\dd v\right\|_{L^2}^2\,+\,\int_\T\eta^{2\al}\,\left|\d_x^3v\right|\,dx\,\lesssim\,
 \left\|\eta\right\|_{L^\infty}^{2\al-2}\,\left\|\big(\d_xv,\d_x\eta\big)\right\|_{L^\infty}^2\,
\left\|\big(\dd v,\dd\eta\big)\right\|_{L^2}^2\,.
\end{equation}

\subsubsection{Estimates for $\dd\eta$} \label{sss:eta-H^2}

Now, we perform an energy estimate on the equation \eqref{eq:dd_eta} for $\dd\eta$: analogously to \eqref{est:en-dd_v}, we find
\begin{equation} \label{est:en-dd_eta}
\frac{1}{2}\,\frac{d}{dt}\left\|\dd\eta\right\|_{L^2}^2\,+\,\int_\T\eta^{2\bt}\,\left|\d_x^3\eta\right|\,dx\,=\,\int_\T g\,\dd\eta\,dx\,.
\end{equation}
Therefore, we have to bound the term on the right-hand side. We will consider the terms $g_j$, for $j\in\{1\ldots 9\}$, one by one.

We start by considering $g_1$. Recall that, by assumption, one has $\beta=1/2$ or $\beta=1$, in which cases $g_1\equiv0$, or $\beta\geq 3/2$. Then, arguing
exactly as done for the $f_1$ term, in the previous paragraph, we find
\[
\left|\int_\T g_1\,\dd\eta\,dx\right|\,\lesssim\,\left\|\eta\right\|_{L^\infty}^{2\bt-2}\,\left\|\d_x\eta\right\|_{L^\infty}^2\,
\left\|\dd\eta\right\|_{L^2}^2\,.
\]

The treatment of the term involving $g_5$ is also similar. We recall that also $g_5$ is possibly zero, depending on the value of $\alpha$.
Assuming to be in the worst situation, namely that $\al\geq3/2$ and then $g_5\neq0$, we can bound
\begin{align*}
\left|\int_\T g_5\,\dd\eta\,dx\right|\,\lesssim\,
\left\|\eta\right\|_{L^\infty}^{2\al-3}\,\left\|\d_xv\right\|_{L^\infty}^2\,
\left\|\d_x\eta\right\|_{L^4}^2\,\left\|\dd\eta\right\|_{L^2}\;\lesssim\;\left\|\eta\right\|_{L^\infty}^{2\al-2}\,\left\|\d_xv\right\|_{L^\infty}^2\,
\left\|\dd\eta\right\|_{L^2}^2\,.
\end{align*}

The terms $g_2$, $g_6$ and $g_7$ can be estimated in a similar way, see also what done for $f_2$ and $f_3$. We get
\begin{align*}
\left|\int_\T\big(g_2+g_6+g_7\big)\,\dd\eta\,dx\right|\,&\lesssim\,\left(\left\|\eta\right\|_{L^\infty}^{2\al-2}\,+\,\left\|\eta\right\|_{L^\infty}^{2\bt-2}\right)
\,\left\|\big(\d_xv,\d_x\eta\big)\right\|_{L^\infty}^2\,\left\|\big(\dd v,\dd\eta\big)\right\|_{L^2}^2\,,
\end{align*}

Also the control of the terms $g_4$ and $g_9$ is easy, as it is analogous to what already done for $f_6$. One has the following estimate,
\begin{align*}
\left|\int_\T g_4\,\dd\eta\,dx\right|\,&\lesssim\,\left\|\eta\right\|_{L^\infty}^{\bt-1}\,\left\|\d_x\eta\right\|_{L^\infty}\,\|\dd\eta\|_{L^2}\,
\left\|\eta^\bt\,\d_x^3\eta\right\|_{L^2} \\
&\leq\,C_\de\,\left\|\eta\right\|_{L^\infty}^{2\bt-2}\,\left\|\d_x\eta\right\|^2_{L^\infty}\,\|\dd\eta\|^2_{L^2}\,+\,\de\,
\left\|\eta^\bt\,\d_x^3\eta\right\|^2_{L^2}\,,
\end{align*}
where, as before, $\de>0$ can be chosen arbitrarily small, and the constant $C_\de>0$ only depends on the value of $\de$,
and, \tsl{mutatis mutandis}, also the following one,
\begin{align*}
\left|\int_\T g_9\,\dd\eta\,dx\right|\,
&\leq\,C_\de\,\left\|\eta\right\|_{L^\infty}^{2\al-2}\,\left\|\d_xv\right\|^2_{L^\infty}\,\|\dd\eta\|^2_{L^2}\,+\,\de\,
\left\|\eta^\al\,\d_x^3v\right\|^2_{L^2}\,.
\end{align*}

Let us now focus on the term $g_3$. We proceed analogously as done for $f_5$, namely we integrate by parts once and we reconduct ourselves to treat
terms of the same type as the previous ones. As a matter of fact, we see that
\[
\int_\T \eta^{2\bt-1}\,\big(\dd\eta\big)^3\,dx\,=\,-\,(2\bt-1)\int_\T\eta^{2\bt-2}\,\big(\d_x\eta\big)^2\,\big(\dd\eta\big)^2\,dx\,-\,
2\int_\T\eta^{2\bt-1}\,\d_x\eta\,\dd\eta\,\d_x^3\eta\,dx\,,
\]
from which we easily deduce that
\begin{align*}
\left|\int_\T g_3\,\dd\eta\,dx\right|\,&\leq\,
C_\de\,\left\|\eta\right\|_{L^\infty}^{2\bt-2}\,\left\|\d_x\eta\right\|^2_{L^\infty}\,\|\dd\eta\|^2_{L^2}\,+\,\de\,
\left\|\eta^\bt\,\d_x^3\eta\right\|^2_{L^2}\,,
\end{align*}
where $\de>0$ is as above and $C_\de>0$ is a possibly different constant, but which still depends only on $\beta$ and $\de$.

It remains us to deal with $g_8$, but this term is exactly equal (up to constant factors) to the term $f_5$ treated in Paragraph \ref{sss:v-H^2}.
Therefore, we gather
\begin{align*}
\left|\int_\T g_8\,\dd\eta\,dx\right|\,&\leq\,C_\de\,\left\|\eta\right\|_{L^\infty}^{2\al-2}\,\left\|\d_x\eta\right\|_{L^\infty}^2\,
\left\|\dd v\right\|_{L^2}^2\,+\,\de\,\left\|\eta^\alpha\,\d_x^3v\right\|^2_{L^2}\,.
\end{align*}

Plugging all the previous bounds into \eqref{est:en-dd_eta} and taking $\de>0$ small enough, we finally find
\begin{equation} \label{est:dd_eta_final}
\frac{d}{dt}\left\|\dd\eta\right\|_{L^2}^2\,+\,\int_\T\eta^{2\bt}\,\left|\d_x^3\eta\right|\,dx\,\lesssim\,
 \left(\left\|\eta\right\|_{L^\infty}^{2\al-2}\,+\,\left\|\eta\right\|_{L^\infty}^{2\bt-2}\right)\,\left\|\big(\d_xv,\d_x\eta\big)\right\|_{L^\infty}^2\,
\left\|\big(\dd v,\dd\eta\big)\right\|_{L^2}^2\,.
\end{equation}

\subsection{Final estimates} \label{ss:final}

We are now ready to find a bound for the norm of the solution in terms of the norm of the initial datum only, in some small time interval $[0,T]$.

As in \cite{F-GB}, let us introduce, for any time $t\geq0$, the function
\[
E(t)\,:=\,\left\|v(t)\right\|_{H^2}^2\,+\,\left\|\eta(t)\right\|_{H^2}^2
\]
as the energy of the solution $\big(v,\eta\big)$.
Let also $E(0)$ be defined as the same quantity, but computed on the initial datum $\big(v_0,\eta_0\big)$.

Let us also define
\begin{align*}
 E_0(t)\,:=\,\left\|v(t)\right\|_{L^2}^2\,+\,\left\|\eta(t)\right\|_{L^2}^2 \qquad \mbox{ and }\qquad
 E_2(t)\,:=\,\left\|\dd v(t)\right\|_{L^2}^2\,+\,\left\|\dd\eta(t)\right\|_{L^2}^2\,.
\end{align*}
As before, we denote by $E_0(0)$ and $E_2(0)$ the same quantities, when computed on the initial datum.
Notice that, as observed at the beginning of this Subsection \ref{ss:high}, one has
\begin{equation} \label{eq:en_equiv}
\forall\,t\geq0\,,\qquad\qquad E(t)\,\sim\,E_0(t)\,+\,E_2(t)\,.
\end{equation}
Therefore, in order to find an inequality of the type
\begin{equation} \label{est:en_to-prove}
\sup_{t\in[0,T]}E(t)\,\leq\,C\,E(0)\,,
\end{equation}
for some time $T>0$ and some suitable constant $C>0$, it is enough to find a similar bound (in terms of the whole energy $E(0)$
of the initial datum) for the functions $E_0(t)$ and $E_2(t)$.

Let us start with the low regularity norm. Its control is very easy: owing to inequalities \eqref{est:v-L^2} and \eqref{est:eta_L^2},
one simply has
\begin{equation} \label{est:E_0}
\forall\,t\geq0\,,\qquad\qquad\quad E_0(t)\,\leq\,2\,E_0(0)\,.
\end{equation}

We now switch to the bound for $E_2(t)$. Summing up inequalities \eqref{est:dd_v_final} and \eqref{est:dd_eta_final}, we infer
that, for any time $t\geq0$, one has
\[
\frac{d}{dt}E_2(t)\,\lesssim\,A(t)\,E_2(t)\,,\qquad\qquad \mbox{ where }\qquad
A(t)\,:=\, \left(\left\|\eta\right\|_{L^\infty}^{2\al-2}\,+\,\left\|\eta\right\|_{L^\infty}^{2\bt-2}\right)\,\left\|\big(\d_xv,\d_x\eta\big)\right\|_{L^\infty}^2\,.
\]
On the other hand, using repeatedly \eqref{est:interp_2} with $p=2$ and \eqref{est:H2-interp}, we see that
\[
A(t)\,\lesssim\,\big(E(t)\big)^\g\,,\qquad\qquad \mbox{ for some large enough }\ \g>0\,.
\]
Therefore, integrating the previous inequality in time and summing up the resulting expression with \eqref{est:E_0}, we find, for any
time $t\geq0$, the estimate
\begin{equation*}
 E(t)\,\leq\,E(0)\,+\,C\int^t_0\big(E(\t)\big)^{\g+1}\,d\t\,,
\end{equation*}
for a suitable constant $C>0$. From this inequality, standard arguments ensure that there exist a time $T>0$ and a (possibly larger) constant $C>0$
such that the bound \eqref{est:en_to-prove} holds true.

\medbreak
This completes the proof of the \tsl{a priori} bounds for system \eqref{eq:toy} at $H^2$ level of regularity.
The propagation of $H^m$ regularity, for integer values $m\geq3$, can be obtained by very similar arguments.
As already said, we omit here the details of the proof
of the existence and uniqueness of solutions.

The proof of Theorem \ref{t:wp} is thus completed.

\subsection{Further considerations} \label{ss:general}
The assumptions on $\alpha$ and $\beta$ stated in Theorem \ref{t:wp} may look quite stringent. However, in certain sense, $\eta=\sqrt{\omega}$ seems sharp in our analysis.

To see this, let us consider the powers $\s_n\,:=\,\omega^{1/n}$, for $n\geq2$. Notice that,
in particular, when taking $n=2$, we receover the function $\s_2\,=\,\sqrt{\o}\,=\,\eta$ as before. We are going to show that a series of natural requirements
will appear for $n$ in order to close the estimates, and that the best possible choice in order to respect all those constraints
and, at the same time, consider the largest intervals for $\al$ and $\bt$ is to take $n=2$.

Let us focus on \tsl{a priori} bounds only. First of all, inequality \eqref{est:v-L^2} still holds true. On the other hand, from \eqref{est:o_L^1} we deduce that
\begin{equation} \label{est:sigma_L^2}
\forall\,t\geq0\,,\qquad\qquad
\|\s(t)\|_{L^n}^n\,\lesssim\,\left\|\s_0\right\|_{L^n}^n\,+\,\left\|v_0\right\|_{L^2}^2\,.
\end{equation}

Next, we recast system \eqref{eq:toy} in terms of $\big(v,\s\big)$: we get
\begin{equation} \label{eq:v-s}
\left\{\begin{array}{l}
       \d_tv\,-\,\d_x\Big(\s^{n\alpha}\,\d_xv\Big)\,=\,0 \\[1ex]
       \d_t\s\,-\,\d_x\Big(\s^{n\bt}\,\d_x\s\Big)\,=\,\dfrac{1}{n}\,\s^{n(\al-1)+1}\,\big|\d_xv\big|^2\,+\,(n-1)\,\s^{n\bt-1}\,\big|\d_x\s\big|^2\,.
       \end{array}
\right.
\end{equation}
Differentiating twice the equation for $v$ with respect to the space variable yields
\begin{align*}
&\d_t\dd v\,-\,\d_x\Big(\s^{n\al}\,\d_x\dd v\Big) \\ 
&\qquad =\,n\al\,(n\al-1)\,(n\al-2)\,\s^{n\al-3}\,\big(\d_x\s\big)^3\,\d_xv\,+\,2n\al\,(n\al-1)\,\s^{n\al-2}\,\d_x\s\,\dd\s\,\d_xv \\
&\qquad\qquad +\,3n\al\,(n\al-1)\,\s^{n\al-2}\,\big(\d_x\s\big)^2\,\dd v\,+\,n\al\,\d_x\Big(\s^{n\al-1}\,\dd\s\,\d_xv\Big) \\
&\qquad\qquad\qquad +\,2n\al\,\s^{n\al-1}\,\dd\s\,\dd v\,+\,2n\al\,\s^{n\al-1}\,\d_x\s\d_x^3v\,.
\end{align*}
Keeping in mind the computations of the previous subsection (for the case $n=2$), it is apparent
that the most dangerous terms are the first one and the last one appearing in the right-hand side of the previous expression.

On the one hand, in order to control the last term, namely $2n\al\,\s^{n\al-1}\,\d_x\s\d_x^3v$, when performing an energy estimates
against $\dd v$, it is clear that we need to put a power $\s^{n\al/2}$ together with $\d_x^3v$ and use the effects of the degenerate viscosity. In turn,
this would require the condition $n\al/2\,-\,1\geq0$, \tsl{i.e.} $n\al\geq2$. Notice that, when performing estimates at the $H^m$ level of regularity,
one would rather need to consider a term of the form
\[
2n\al\,\s^{n\al-1}\,\d_x\s\d_x^{m+1}v\,,
\]
which requires again the same condition $n\al\geq2$ to be dealt with.

On the other hand, in order to control the term $n\al\,(n\al-1)\,(n\al-2)\,\s^{n\al-3}\,\big(\d_x\s\big)^3\,\d_xv$, one simply needs that
either that term vanishes, which implies $n\al\in\big\{1,2\big\}$, or that the power of $\s$ is positive, namely $n\al\geq3$. When considering
general $H^m$ estimates, instead, a simple induction argument shows that the corresponding term has the form
\[
 n\al\,(n\al-1)\,(n\al-2)\cdots(n\al-m)\,\s^{n\al-(m+1)}\,\big(\d_x\s\big)^{m+1}\,\d_xv\,,
\]
which, to be controlled, requires either $n\al\in\big\{1,2,\ldots m\big\}$, or $n\al\geq m+1$.

When computing the equation for $\dd\s$, we will get all the terms appearing above (where $\bt$ and $\s$ will replace, respectively, $\al$ and $v$),
which come from the commutator $\d_x\big([\s^{n\bt},\dd]\,\d_x\s\big)$ of the parabolic part of the equation. However,
we have to pay attention also to the terms arising from the differentiation of the right-hand side of the second equation in \eqref{eq:v-s}.

First of all, when applying the operator $\d_x^m$ to the term $(n-1)\,\s^{n\bt-1}\,\big|\d_x\s\big|^2$, the worst terms appear when
the operator acts onto the factor $\s^{n\bt-1}$, giving rise to
\begin{equation} \label{eq:bad_1}
(n-1)\,(n\bt-1)\,(n\bt-2)\cdots(n\bt-m)\,\s^{n\bt-(m+1)}\,\big|\d_x\s\big|^2\,,
\end{equation}
and when the operator acts always on the same factor $\d_x\s$, yielding a term of the form
\begin{equation} \label{eq:bad_2}
2\,(n-1)\,\s^{n\bt-1}\,\d_x\s\,\d_x^{m+1}\s\,.
\end{equation}
In order to control the term appearing in \eqref{eq:bad_1}, we need the condition $n\bt\in\big\{1,2,\ldots m\big\}$ or $n\bt\geq m+1$, whereas
for the term in \eqref{eq:bad_2} we need to use diffusion to absorbe the higher order term $\d_x^{m+1}\s$, giving rise to the condition $n\bt\geq2$.

Finally, we deal with the term $\s^{n(\al-1)+1}\,\big|\d_xv\big|^2$ appearing on the right-hand side of the equation for $\s$.
Repeating \tsl{mutatis mutandis} the previous reasonning, after applying the $\d_x^{m}$ operator to the equation, we have to deal with
terms of the form
\begin{align*}
(n\al-n+1)\,(n\al-n)\cdots\big(n\al-n-(m-2)\big)\,&\s^{n(\al-1)-(m-1)}\,\big(\d_x\s\big)^m\,|\d_xv|^2 \\
&\qquad\mbox{ and }\quad\qquad\qquad
\s^{n(\al-1)+1}\,\d_xv\,\d_x^{m+1}v
\end{align*}
plus of course other terms, which are however of lower order. The term
$$
\s^{n(\al-1)+1}\,\d_xv\,\d_x^{m+1}v
$$
is the most dangerous one. For this term, we have to use again viscous effects, which in turn requires that
\begin{equation} \label{eq:bad-a}
\frac{n\al}{2}\,-\,n\,+\,1\,\geq0\qquad\qquad \Longrightarrow\qquad\qquad \al\,\geq\,2\,\left(1\,-\,\frac{1}{n}\right)\,.
\end{equation}
This requires that $\alpha\geq1$, being the case $n=2$ the best case possible in terms of the size of $\alpha$. For controlling the former one, instead, we need to have
\[
n\al\,\in\,\Big\{n-1\,,\,n\,,\,n+1\,,\ldots\,n+m-2\Big\}\,,\qquad\qquad \mbox{ or }\qquad
n\al\,\geq\,n+m-1\,.
\]

To conclude this part, we observe that, if we renounce to a certain symmetry of conditions to be imposed on $\al$ and $\bt$, the fact of taking $n\geq3$ imposes
more severe requirements on $\al$, keep in mind \eqref{eq:bad-a}, but allows to go much below the threshold value $1$ for the parameter $\bt$.

\section{Finite time blow-up} \label{s:blow-up}

We now present the finite time blow-up of smooth solutions to \eqref{eq:toy}, under suitable conditions
on the parameters $\alpha$ and $\beta$ and on the symmetry of the initial data.



\begin{thm} \label{t:symmetry}
Let $\alpha\geq1$ and $\beta\geq1$ be two real parameters.
Let $\big(v_0,\o_0\big)$ be a couple of functions belonging to $H^5(\T)$ such that $\o_0\geq0$, 
and let $\big(v,\o\big)$ be the related $H^5$ solution constructed in Theorems \ref{t:wp}.
Let the following conditions be in force:
\begin{enumerate}[(i)]
 \item $v_0$ is odd with respect to $0$, while $\o_0$ is even with respect to $0$;
 \item $\o_0(0)=0$, with $\d_x^2\o_0(0)>0$;
\end{enumerate}
In addition, assume one of the following two conditions:
\begin{itemize}
 \item $\bt=1$, in which case take also $\alpha=1$ or $\alpha\geq2$;
 \item $\bt=2$ or $\beta\geq3$, and then take $\alpha=1$; in this case, assume moreover that
\[
\d_x v_0(0)\,\geq\,\sqrt{2\,\d_x^2\o_0(0)}\,.
\]
\end{itemize}

Then, there exists a time $t_0>0$ such that, if the solution $\big(v,\o\big)$ does not blow up first at a different time and place, one has
\[
\lim_{t\ra t_0^-}\d_x^2 \o(t,0)\,=\,+\infty\,.
\]
\end{thm}

Before proving the previous result, some remarks are in order.

\begin{rmk} \label{r:blow-up}
In fact, we are going to prove that, when $\alpha=1$ (and then $\beta\geq 3$ or $\bt\in\{1,2\}$),
then\footnote{Recall that, in the case $\bt>1$, one also needs the additional condition linking $\d_xv_0(0)$ and $\dd\o_0(0)$.}
$\dd\o(t,0)$ blows up in finite time if and only if $\d_xv(t,0)$ also blows up at the same time.

In the case $\bt=1$ and $\alpha\geq2$, instead, only $\dd\o(t,0)$ blows up, while $\d_xv(t,0)$ remains bounded (in fact, it remains constant, equal
to $\d_xv_0(0)$ at any time).

In this sense, Theorem \ref{t:symmetry} strengthen the result of \cite{F-GB}, inasmuch as it shows that the blow-up is not only of Burgers type,
but really relies on a different mechanism. In particular, the blow up in the case $\bt=1$ and $\alpha\geq2$ is not due to compressibility of the velocity field.
\end{rmk}

We now present the proof of Theorem \ref{t:symmetry}.

\begin{proof}[Proof of Theorem \ref{t:symmetry}]
The proof follows the main lines of the method introduced in \cite{Bae-Gra} (see also \cite{Jeong-Kang}). Namely, we will prove that, assuming that the
solution exists and remains smooth, then there is a finite time singularity happening at $(t,x)=(t_0,0)$.

To begin with, 
we introduce the following notation: for any $j\in\N$, we set
$$
V^{(j)}(t)\,:=\,\left(\d_x^jv\right)(t,0)\qquad\qquad \text{ and }\qquad\qquad \Omega^{(j)}(t)\,:=\,\left(\d_x^j\o\right)(t,0)\,.
$$

Next, we observe that we have supposed $v_0$ to be odd and $\o_0$ to be even with respect to the origin, and such symmetries are preserved by the evolution
for strong solutions. As a consequence, we obtain that
\[
\forall\,t\geq0\,,\qquad\qquad\qquad
V^{(0)}(t)\,=\,0\qquad\qquad \mbox{ and } \qquad\qquad V^{(2)}(t)\,=\,0\,.
\]
together with the corresponding counterparts for the function $\o$:
\[
\forall\,t\geq0\,,\qquad\qquad\qquad
\Omega^{(1)}(t)\,=\,0\qquad\qquad \mbox{ and } \qquad\qquad \Omega^{(3)}(t)\,=\,0\,.
\]

Now, let us compute the equation for $\o$ at the point $(t,0)$, for any time $t\geq0$: we have
\begin{align*}
\frac{d}{dt}\Omega^{(0)}\,&=\,\kappa_0\,\beta\left(\Omega^{(0)}\right)^{\beta-1}\left(\Omega^{(1)}\right)^2\,+\,
\kappa_0\left(\Omega^{(0)}\right)^\beta\,\Omega^{(2)}\,+\,\left(\Omega^{(0)}\right)^\alpha\,\big|V^{(1)}\big|^2 \\
&=\,\kappa_0\,\left(\Omega^{(0)}\right)^\beta\,\Omega^{(2)}\,+\,\left(\Omega^{(0)}\right)^\alpha\,\big|V^{(1)}\big|^2\,,
\end{align*}
where we have used that $\Omega^{(1)}(t)=0$ for any $t\geq0$. Integrating the previous ODE, we get
\[
\forall\,t\geq0\,,\qquad 
\Omega^{(0)}(t)\,=\,\Omega^{(0)}(0)\,\exp\left(\int_0^t\kappa_0\left(\Omega^{(0)}\right)^{\beta-1}\,\Omega^{(2)}\,+\,
\left(\Omega^{(0)}\right)^{\alpha-1}\,\big|V^{(1)}\big|^2d\t\right)\,,
\]
from which we conclude that
\begin{equation} \label{eq:omega-0}
\forall\,t\geq0\,,\qquad\qquad\qquad \Omega^{(0)}(t)\,=\,0\,.
\end{equation}

At this point, we differentiate the equation for $v$ with respect to $x$ to get
\begin{align*}
\d_t\d_xv\,=\,\o^\alpha\,\d_x^3v\,+\,2\,\alpha\,\o^{\alpha-1}\,\d_x\o\,\d_x^2v\,+\,\alpha\,(\alpha-1)\,\o^{\alpha-2}\,\big(\d_x\o\big)^2\,\d_xv\,+\,
\alpha\,\o^{\alpha-1}\,\d_x^2\o\,\d_xv\,.
\end{align*}
Taking the value of the previous expression at the point $x=0$ and using \eqref{eq:omega-0} and the symmetry properties of the solution $\big(v,\o\big)$, we obtain
\begin{equation} \label{eq:V^1}
\frac{d}{dt}V^{(1)}\,=\,\alpha\,\left(\Omega^{(0)}\right)^{\alpha-1}\,\Omega^{(2)}\,V^{(1)}\,=\,
\left\{\begin{array}{lcl}
        \Omega^{(2)}\,V^{(1)} \qquad & \mbox{ if } \ & \alpha=1 \\[1ex]
        0 \qquad & \mbox{ if } \ & \alpha\geq 2\,.
       \end{array}
\right.
\end{equation}

Similarly, long but straightforward computations yield
\begin{align*}
\d_t\dd\o\,&=\,\k_0\,\o^\bt\,\d_x^4\o\,+\,4\,\k_0\,\bt\,\o^{\bt-1}\,\d_x\o\,\d_x^3\o\,+\,3\,\k_0\,\bt\,\o^{\bt-1}\,\big(\dd\o\big)^2 \\
&\qquad +\,6\,\k_0\,\bt\,(\bt-1)\,\o^{\bt-2}\,\big(\d_x\o\big)^2\,\dd\o\,+\,\k_0\bt\,(\bt-1)\,(\bt-2)\,\o^{\bt-3}\,\big(\d_x\o\big)^4 \\
&\qquad\qquad\qquad +\,\al\,(\al-1)\,\o^{\al-2}\,\big(\d_x\o\big)^2\,\big(\d_xv\big)^2\,+\,\al\,\o^{\al-1}\,\dd\o\,\big(\d_xv\big)^2 \\
&\qquad\qquad\qquad\qquad\qquad +\,
4\,\al\,\o^{\al-1}\,\d_x\o\,\d_xv\,\dd v\,+\,2\,\o^\al\,\big(\dd v\big)^2\,+\,2\,\o^\al\,\d_xv\,\d_x^3v\,.
\end{align*}
Then, owing again to the symmetry properties of the solution and to \eqref{eq:omega-0}, we infer
\begin{equation} \label{eq:O^2}
\frac{d}{dt}\Omega^{(2)}\,=\,3\,\k_0\,\bt\,\left(\Omega^{(0)}\right)^{\bt-1}\,\left(\Omega^{(2)}\right)^2\,+\,
\alpha\,\left(\Omega^{(0)}\right)^{\alpha-1}\,\Omega^{(2)}\,\left(V^{(1)}\right)^2\,.
\end{equation}

We stress here that the previous result holds true for any $\alpha\geq1$ and $\beta\geq1$ such that:
\begin{itemize}
 \item $\alpha=1$, or $\alpha\geq2$;
 \item $\bt\in\{1,2\}$, or $\beta\geq3$.
\end{itemize}
Of course, depending on the precise values of those parameters, one can simplify further the previous expression, owing to \eqref{eq:omega-0}.

To complete our proof, we now consider three different cases.

\medbreak
\noindent \tit{Case 1: $\alpha=1$ and $\big[\bt=2$ or $\beta\geq3\big]$.}

In this case, putting \eqref{eq:V^1} and \eqref{eq:O^2} together, we find that
\[
\frac{d}{dt}V^{(1)}\,=\,\Omega^{(2)}\,V^{(1)}\qquad\qquad \mbox{ and }\qquad\qquad 
\frac{d}{dt}\Omega^{(2)}\,=\,\Omega^{(2)}\,\left(V^{(1)}\right)^2\,.
\]
In particular, we deduce the inequality
\[
\frac{d}{dt}\Omega^{(2)}\,=\,\frac{1}{2}\,\frac{d}{dt}\left(V^{(1)}\right)^2\qquad\quad\Longrightarrow\qquad\quad 
\Omega^{(2)}(t)\,-\,\Omega^{(2)}(0)\,=\,\frac{1}{2}\,\left(\left(V^{(1)}(t)\right)^2-\left(V^{(1)}(0)\right)^2\right)\,,
\]
which immediately implies that $V^{(1)}$ blows up at some time $t=t_0$ if and only if $\Omega^{(2)}$ blows up at the same time $t=t_0$.

To see that the blow-up indeed occurs at some time $t=t_0$, we use the previous equation to write that
\[
\frac{d}{dt}\Omega^{(2)}\,=\,\Omega^{(2)}\,\left(V^{(1)}\right)^2\,=\,
2\,\left(\Omega^{(2)}\right)^2\,+\,\Omega^{(2)}\,\left(\left(V^{(1)}(0)\right)^2\,-\,2\,\Omega^{(2)}(0)\right)\,\geq\,
2\,\left(\Omega^{(2)}\right)^2\,,
\]
where we have used that $\Omega^{(2)}\geq0$, because $x=0$ is a point of minimum for the function $\o(t)$, together with the assumption
\[
\d_x v_0(0)\,\geq\,\sqrt{2\,\d_x^2\o_0(0)}\,.
\]
The previous differential relation easily implies the claimed blow-up in finite time.

\medbreak
\noindent \tit{Case 2: $\alpha=\beta=1$.}

This time, equations \eqref{eq:V^1} and \eqref{eq:O^2} become
\[
\frac{d}{dt}V^{(1)}\,=\,\Omega^{(2)}\,V^{(1)}\qquad\qquad \mbox{ and }\qquad\qquad 
\frac{d}{dt}\Omega^{(2)}\,=\,3\,\k_0\,\left(\Omega^{(2)}\right)^2\,+\,\Omega^{(2)}\,\left(V^{(1)}\right)^2\,.
\]

First of all, we observe that, in particular, one has
\[
\frac{d}{dt}\Omega^{(2)}\,\geq\,\Omega^{(2)}\,\left(V^{(1)}\right)^2\,.
\]
Thus, arguing as above, we see that a blow-up of $V^{(1)}$ at time $t=t_0$ implies a blow-up, at the same time, also of the quantity $\Omega^{(2)}$.

On the other hand, as $x=0$ is a point of minimum of the function $\o(t,\cdot)$, we deduce that $\Omega^{(2)}\geq0$, which implies that
\[
 \frac{d}{dt}\Omega^{(2)}\,\geq\,3\,\k_0\,\left(\Omega^{(2)}\right)^2\,,
\]
and this, in turn, yields finite time blow-up of the quantity $\Omega^{(2)}$, with rate $O\big((t_0-t)^{-1}\big)$:
\begin{equation} \label{est:O_blow-up}
\Omega^{(2)}(t)\,\geq\,\frac{\Omega^{(2)}(0)}{1\,-\,3\,\k_0\,\Omega^{(2)}(0)\,t}\,.
\end{equation}

On the other hand, solving explicitly the equation for $V^{(1)}$, we gather
\[
V^{(1)}(t)\,=\,V^{(1)}(0)\,\exp\left(\int^t_0\Omega^{(2)}(\t)\,d\t\right)\,.
\]
Then, estimate \eqref{est:O_blow-up} implies that also $V^{(1)}$ must blow up at the same time $t=t_0$ at which $\Omega^{(2)}$ blows up.

\medbreak
\noindent \tit{Case 3: $\alpha\geq2$ and $\beta=1$.}

Finally, we consider the case $\alpha\geq 2$ and $\bt=1$. In this case, owing to \eqref{eq:omega-0}, relations \eqref{eq:V^1} and \eqref{eq:O^2} become
\[
\frac{d}{dt}V^{(1)}\,=\,0\qquad\qquad \mbox{ and }\qquad\qquad 
\frac{d}{dt}\Omega^{(2)}\,=\,3\,\k_0\,\left(\Omega^{(2)}\right)^2\,.
\]
Thus, we easily see that $V^{(1)}(t)\,=\,V^{(1)}(0)$ remains bounded (constant, in fact) during the evolution. On the other hand,
exactly as before, from the second equation we gather that
\[
\Omega^{(2)}(t)\,=\,\frac{\Omega^{(2)}(0)}{1\,-\,3\,\k_0\,\Omega^{(2)}(0)\,t}\,\longrightarrow\,+\infty\qquad\qquad \mbox{ when }\qquad t\,\ra\, t_0^-\,,
\]
where we have set $t_0\,:=\,\big(3\,\k_0\,\Omega^{(2)}(0)\big)^{-1}$.
\end{proof}

\section{The convective case} \label{s:convective}

In this section, we discuss the case when we add to system \eqref{eq:toy} two convective terms, one in each equation. The equations then become
\begin{equation} \label{eq:parab}
\left\{\begin{array}{l}
       \d_tv\,+\,v\,\d_xv\,-\,\d_x\left(\o^\alpha\,\d_xv\right)\,=\,0 \\[1ex]
       \d_t\o\,+\,v\,\d_x\o\,-\,\kappa_0\d_x\left(\o^\beta\,\d_x\o\right)\,=\,\o^\alpha\,\big|\d_xv\big|^2\,,
       \end{array}
\right.
\end{equation}
supplemented, as before, with initial conditions $\big(v,\o\big)_{|t=0}\,=\,\big(v_0,\o_0\big)$, for some $\o_0\geq0$.
In our modest opinion, the inclusion of the convection term makes this model closer to problems in connection with turbulence theory.

It goes without saying that the well-posedness theory for system \eqref{eq:toy}, presented in Section \ref{s:wp}, applies also
to system \eqref{eq:parab}, with only minor changes. The control of the convective terms $v\,\d_xv$ and $v\,\d_x\o$ and of their higher order derivatives
is fairly standard; for instance, one can repeat the analysis of \cite{F-GB}. Notice however that, as now $v$ is no more divergence-free
and there is no damping term in system \eqref{eq:parab}, the energy $E_0\,=\,\|v\|^2_{L^2}\,+\,\|\o\|_{L^1}$ may grow in time with exponential rate.

Therefore, our main goal here is to establish finite time blow-up results for system \eqref{eq:parab}, in the same spirit of Theorem \ref{t:symmetry} above.
We present two kind of results. The first one concerns the case when the initial velocity field is assumed to satisfy a ``Burgers-type'' condition: then,
finite time singularity formation can be proved for the whole class of parameters considered in Theorem \ref{t:symmetry} above.
\begin{thm} \label{th:sing_convective}
Let $\alpha\geq1$ and $\beta\geq1$ be two real parameters satisfying
\[
\alpha\,\in\,\{1\}\,\cup\,[2,+\infty[\,\qquad\qquad \mbox{ and }\qquad\qquad \beta\,\in\,\{1,2\}\,\cup\,[3,+\infty[\,.
\]
Let $\big(v_0,\o_0\big)$ be a couple of functions belonging to $H^5(\Omega)$, with $\o_0\geq0$, 
and let $\big(v,\o\big)$ be the related $H^5$ solution. 
Assume that:
\begin{enumerate}[(i)]
 \item $v_0$ is odd with respect to $0$, while $\o_0$ is even with respect to $0$;
 \item $\o_0(0)=0$, with $\d_x^2\o_0(0)>0$;
 \item $\d_xv_0(0)<0$.
\end{enumerate}

Then, there exists a time $t_0>0$ such that the solution $\big(v,\o\big)$ blows up at time $t=t_0$ for $x=0$: specifically, one has either
\[
\lim_{t\ra t_0^-}\d_xv(t,0)\,=\,-\infty\qquad\qquad \mbox{ or }\qquad\qquad  \lim_{t\ra t_0^-}\dd\o(t,0)\,=\,+\infty\,.
\]
\end{thm}

\begin{proof}
The proof is simple. First of all, we observe that the system with convective terms \eqref{eq:parab} preserves the initial positivity of the function $\o$:
one has $\o(t,x)\geq0$ for any $t\geq0$ and $x\in\T$. By repeating the computations performed in the proof of Theorem \ref{t:symmetry}, we see also
that $\o(t,0)=0$ for any $t\geq0$.

Next, we see that the convective term does not alter the symmetry of the initial data: therefore, for any $t\geq0$ one has that
$v(t)$ remains odd with respect to the origin, whereas $\o(t)$ remains even with respect to the origin.

At this point, we compute the equations for $\d_xv$ and for $\dd\o$. For the former term, it is easy to see that the convective term gives an additional contribution
on the right-hand side of the form
\[
-\,\d_x\big(v\,\d_xv\big)\,=\,-\,v\,\dd v\,-\,(\d_xv)^2\,.
\]
Therefore, computing the resulting expression at the point $x=0$, and adopting the same notation as in the proof of Theorem \ref{t:symmetry}, we find
\begin{equation} \label{eq:V^1_conv}
\frac{d}{dt}V^{(1)}\,=\,\alpha\,\left(\Omega^{(0)}\right)^{\alpha-1}\,\Omega^{(2)}\,V^{(1)}\,-\,\left(V^{(1)}\right)^2\,,
\end{equation}
which takes the place of equation \eqref{eq:V^1}. On the other hand, in the equation for $\dd\o$, the presence of the convective term yields
\[
-\dd\big(v\,\d_x\o\big)\,=\,-\,\left(v\,\d_x^3\o\,+\,2\,\d_xv\,\dd\o\,+\,\dd v\,\d_x\o\right)
\]
on the right-hand side of the expression computed in the previous proof. Then, equation \eqref{eq:O^2} becomes
\begin{equation} \label{eq:O^2_conv}
\frac{d}{dt}\Omega^{(2)}\,=\,3\,\k_0\,\bt\,\left(\Omega^{(0)}\right)^{\bt-1}\,\left(\Omega^{(2)}\right)^2\,+\,
\alpha\,\left(\Omega^{(0)}\right)^{\alpha-1}\,\Omega^{(2)}\,\left(V^{(1)}\right)^2\,-\,2\,V^{(1)}\,\Omega^{(2)}\,.
\end{equation}

With \eqref{eq:V^1_conv} and \eqref{eq:O^2_conv} at hand, we proceed to consider two different cases.

\medbreak
\noindent \tit{Case 1: $\alpha\geq2$.}

In this case, \eqref{eq:V^1_conv} becomes
\begin{equation} \label{eq:V^1_explicit}
\frac{d}{dt}V^{(1)}\,=\,-\,\left(V^{(1)}\right)^2\qquad\qquad \Longrightarrow\qquad\qquad 
V^{(1)}(t)\,=\,\frac{V^{(1)}(0)}{1\,+\,V^{(1)}(0)\,t}\,,
\end{equation}
which of course diverges to $-\infty$ for $t\ra t_0^{-}$, with $t_0\,=\,-\,1/V^{(1)}(0)$. We now claim that this implies explosion also of the quantity
$\Omega^{(2)}(t)$ at the same time $t_0$.

Indeed, when $\beta\geq2$ (recall that, by assumption, this means $\bt=2$ or $\bt\geq3$), equation \eqref{eq:O^2_conv} simply becomes
\[
\frac{d}{dt}\Omega^{(2)}\,=\,-\,2\,V^{(1)}\,\Omega^{(2)}\qquad\qquad \Longrightarrow\qquad\qquad 
\Omega^{(2)}(t)\,=\,\Omega^{(2)}(0)\,\exp\left(-\,2\int^t_0V^{(1)}(\t)\,d\t\right)\,,
\]
which is easily seen to diverge to $+\infty$ for $t\ra t_0^-$, by using the previous explicit expression for $V^{(1)}$ in \eqref{eq:V^1_explicit}.

On the other hand, if $\bt=1$, we have
\[
\frac{d}{dt}\Omega^{(2)}\,=\,3\,\k_0\,\left(\Omega^{(2)}\right)^2\,-\,2\,V^{(1)}\,\Omega^{(2)}\,\geq 3\,\k_0\,\left(\Omega^{(2)}\right)^2\,,
\]
because, by \eqref{eq:V^1_explicit}, $V^{(1)}$ remains negative for all times for which the solution exists.
Therefore, we deduce that also $\Omega^{(2)}$ blows up at some time $t_1\approx\big(3\,\k_0\,\Omega^{(2)}(0)\big)^{-1}$.

\medbreak
\noindent \tit{Case 2: $\alpha=1$.} 

In this case, equations \eqref{eq:V^1_conv} and \eqref{eq:O^2_conv} become
\begin{align*}
\frac{d}{dt}V^{(1)}\,&=\,\Big(\Omega^{(2)}\,-\,V^{(1)}\Big)\,V^{(1)} \\
\frac{d}{dt}\Omega^{(2)}\,&=\,3\,\k_0\,\bt\,\left(\Omega^{(0)}\right)^{\bt-1}\,\left(\Omega^{(2)}\right)^2\,+\,
\Omega^{(2)}\,\left(V^{(1)}\right)^2\,-\,2\,V^{(1)}\,\Omega^{(2)}\,.
\end{align*}
From these ODEs, one can see that $V^{(1)}$ and $\Omega^{(2)}$ remain, respectively, negative and positive, at least for short times. Then, by using the ODEs again,
we see that their time derivatives must, respectively, decrease and increase. Therefore, we finally infer that $V^{(1)}$ remains negative for all times, whereas
$\Omega^{(2)}$ remains positive for all times.

By virtue of this property, we see that
\[
\frac{d}{dt}V^{(1)}\,\leq\,-\,\left(V^{(1)}\right)^2\,,
\]
which again implies blow-up at some time $t_0\approx \left(V^{(1)}(0)\right)^{-1}$, and also that
\[
\frac{d}{dt}\Omega^{(2)}\,\geq\,\Omega^{(2)}\,\left(V^{(1)}\right)^2\,,
\]
which, combined with the previous property, implies the unboundedness of $\Omega^{(2)}$ in $[0,t_0[\,$, or possibly in a smaller time interval.
\end{proof}

\begin{rmk} \label{r:conv_expl}
In fact, we have proven that, in general, it is $\d_xv(t,0)$ which blows up first, and this implies the explosion of $\dd\o(t,0)$ at the same time,
apart from the case when $\bt=1$, for which explosion of the latter quantity may take place before the explosion of the former one (depending
on the relative size of the same quantities at time $t=0$).
\end{rmk}

We now present our second result concerning blow-up phenomena for the system with convection, \tsl{i.e.} system \eqref{eq:parab}.
Our main assumption here will be that $\d_xv_0(0)$ is positive. Therefore, we will see that the blow-up mechanism is truly different
from the previous one (which we have called ``of Burgers type'') and, in particular, to the one pointed out in \cite{F-GB}.

The price to pay is that, with our method, we are not able to show blow-up for the whole range of parameters $\alpha$ and $\beta$ as in Theorem \ref{th:sing_convective},
but we have to restrict to some special cases only.

\begin{thm} \label{th:sing_conv-2}
Let $\bt=1$ and $\alpha\in\{1\}\cup[2,+\infty[\,$. 
Let $\big(v_0,\o_0\big)$ be a couple of functions belonging to $H^5(\Omega)$, with $\o_0\geq0$, 
and let $\big(v,\o\big)$ be the related $H^5$ solution. 
Assume that:
\begin{enumerate}[(i)]
 \item $v_0$ is odd with respect to $0$, while $\o_0$ is even with respect to $0$;
 \item $\o_0(0)=0$, with $\d_x^2\o_0(0)>0$;
 \item $\d_xv_0(0)\geq 0$.
\end{enumerate}
In addition:
\begin{itemize}
 \item if $\alpha\geq2$, we assume that $3\,\k_0\,\dd\o_0(0)\,-\,2\,\d_xv_0(0)\,>\,0$;
 \item if $\alpha=1$, we assume that $3\,\k_0\,\dd\o_0(0)\,>\,1$.
\end{itemize}

Then, there exists a time $t_0>0$ such that the solution $\big(v,\o\big)$ blows up at time $t=t_0$ for $x=0$: more precisely, one has
\[
\lim_{t\ra t_0^-}\dd\o(t,0)\,=\,+\infty\,.
\]
\end{thm}

\begin{proof}
The proof follows the main lines of the one of the previous theorem.
Of course, the symmetry properties of the solutions, as well as the positivity of $\o$ do not depend on the sign of $\d_xv_0(0)$,
so they are kept throught the evolution.

In addition, we have that $\Omega^{(0)}(t)$ remains equal to $0$ for all times $t\geq0$ for which the solution exists.
In particular, the point $x=0$ remains a point of minimum for the function $\o(t)$ at any time $t$, so one always has
$\Omega^{(2)}(t)\geq0$ for all times.

Finally, we point out that expressions \eqref{eq:V^1_conv} and \eqref{eq:O^2_conv} still hold true: recalling that we have taken $\bt=1$ here, we have
\begin{align*}
\frac{d}{dt}V^{(1)}\,&=\,\alpha\,\left(\Omega^{(0)}\right)^{\alpha-1}\,\Omega^{(2)}\,V^{(1)}\,-\,\left(V^{(1)}\right)^2 \\
\frac{d}{dt}\Omega^{(2)}\,&=\,3\,\k_0\,\left(\Omega^{(2)}\right)^2\,+\,
\alpha\,\left(\Omega^{(0)}\right)^{\alpha-1}\,\Omega^{(2)}\,\left(V^{(1)}\right)^2\,-\,2\,V^{(1)}\,\Omega^{(2)}\,.
\end{align*}

As usual, we divide our proof into two different cases, depending on the value of $\alpha$.

\medbreak
\noindent \tit{Case 1: $\alpha\geq2$.}

In this case, using the ODE for $V^{(1)}$, we can compute an explicit expression for it:
\[
 V^{(1)}(t)\,=\,\frac{V^{(1)}(0)}{1\,+\,V^{(1)}(0)\,t}\,.
\]
In particular, we deduce that, for all times $t\geq0$ for which the solution exists, one has
\begin{equation} \label{est:V^1}
0\,\leq\,V^{(1)}(t)\,\leq\,V^{(1)}(0)\,.
\end{equation}

On the other hand, when $\alpha\geq2$, the second ODE becomes
\[
\frac{d}{dt}\Omega^{(2)}\,=\,\Big(3\,\k_0\,\Omega^{(2)}\,-\,2\,V^{(1)}\Big)\,\Omega^{(2)}\,,
\]
from which we deduce, using also \eqref{est:V^1} and the positivity of $\Omega^{(2)}$, the inequalities
\[
\frac{d}{dt}\Omega^{(2)}\,\geq\,\Big(3\,\k_0\,\Omega^{(2)}\,-\,2\,V^{(1)}(0)\Big)\,\Omega^{(2)}\,\geq\,
\frac{1}{3\,\k_0}\,\Big(3\,\k_0\,\Omega^{(2)}\,-\,2\,V^{(1)}(0)\Big)^2
\]

Hence, setting $X(t)\,:=\,3\,\k_0\,\Omega^{(2)}(t)\,-\,2\,V^{(1)}(0)$, we see that the function $X$ satisfies
\[
\frac{d}{dt}X\,\geq\,X^2\,,\qquad\qquad \mbox{ with }\qquad X_{|t=0}\,=\,X(0)\,=\,3\,\k_0\,\Omega^{(2)}(0)\,-\,2\,V^{(1)}(0)\,>\,0\,.
\]
Therefore, $X(t)$ must blow up to $+\infty$ for $t\ra t_0^-$, with $t_0\approx \big(X(0)\big)^{-1}$.

\medbreak
\noindent \tit{Case 2: $\alpha=1$.}

In this case, we focus on the equation for $\Omega^{(2)}$ only, which becomes
\begin{align*}
\frac{d}{dt}\Omega^{(2)}\,&=\,3\,\k_0\,\left(\Omega^{(2)}\right)^2\,+\,\Omega^{(2)}\,\left(V^{(1)}\right)^2\,-\,2\,V^{(1)}\,\Omega^{(2)}\,.
\end{align*}
Now, we use that, for any $\lam>0$ and for any $(y,z)\in\R^2$, we can estimate
\[
y\,z\,\leq\,\frac{1}{\lam}\,y^2\,+\,\lam\,z^2\,,
\]
from which we deduce that, for any $\lam>0$, one has
\begin{equation} \label{eq:ODE_O^2}
\frac{d}{dt}\Omega^{(2)}\,\geq\,\Big(3\,\k_0\,-\,\lam\Big)\,\left(\Omega^{(2)}\right)^2\,+\,\left(\Omega^{(2)}\,-\,\frac{1}{\lam}\right)\,\left(V^{(1)}\right)^2\,.
\end{equation}
As, by assumption, one has $\Omega^{(0)}(0)>1/(3\k_0)$, there exists some $\lam_0>0$ such that
\[
\Omega^{(0)}(0)\,>\,\frac{1}{\lam_0}\,>\,\frac{1}{3\,\k_0}\,.
\]
Using that $\lam_0$ in \eqref{eq:ODE_O^2}, we see that the same inequality has to hold also for later times, and then for all times for which the solution
exists.
Therefore, from \eqref{eq:ODE_O^2} again, we infer
\[
 \frac{d}{dt}\Omega^{(2)}\,\geq\,\Big(3\,\k_0\,-\,\lam_{0}\Big)\,\left(\Omega^{(2)}\right)^2\,,
\]
which immediately implies finite time blow-up of $\Omega^{(2)}$.
\end{proof}

\section*{Acknowledgements}

{\small
The authors are very grateful to the anonymous referees for their careful reading and constructive remarks, as well as for pointing out several interesting references
related to the present work.

The work of both authors has been partially supported by the project ``TURB1D -- Reduced models of turbulence'', operated by the French CNRS through the program ``International 
Emerging Actions 2019''.

The work of the first author has been partially supported by the LABEX MILYON (ANR-10-LABX-0070) of Universit\'e de Lyon, within the program ``Investissement d'Avenir''
(ANR-11-IDEX-0007), and by the projects BORDS (ANR-16-CE40-0027-01), SingFlows (ANR-18-CE40-0027) and CRISIS (ANR-20-CE40-0020-01), all operated by the French National Research Agency (ANR).

The work of the second author was supported by the project "Mathematical Analysis of Fluids and Applications" Grant PID2019-109348GA-I00 funded by MCIN/AEI/ 10.13039/501100011033 and acronym "MAFyA". This publication is part of the project PID2019-109348GA-I00 funded by MCIN/ AEI /10.13039/501100011033. This publication is also supported by a 2021 Leonardo Grant for Researchers and Cultural Creators, BBVA Foundation. The BBVA Foundation accepts no responsibility for the opinions, statements, and contents included in the project and/or the results thereof, which are entirely the responsibility of the authors.

}

\end{document}